\documentclass[11 point]{amsart}
\usepackage[english]{babel}
\usepackage{amsthm}
\usepackage{graphicx}
\usepackage{amsfonts}
\usepackage{eufrak}
\usepackage{amsopn}
\usepackage{amsmath, amsfonts, amssymb, amscd, amsthm}
\usepackage{yhmath}
\usepackage{setspace}
\usepackage{mathtools}
\usepackage[all]{xy}
\newtheorem{theorem}{Theorem}[section]
\newtheorem{lemma}[theorem]{Lemma}

\theoremstyle{definition}

\theoremstyle{remark}
\newtheorem{remark}[theorem]{Remark}
\theoremstyle{example}

\theoremstyle{note}

\numberwithin{equation}{section}

\DeclareMathOperator{\GL}{GL}
\DeclareMathOperator{\SL}{SL}

\DeclareMathOperator{\ind}{ind}

\DeclareMathOperator{\Sp}{Sp}
\DeclareMathOperator{\Gsp}{Gsp}

\DeclareMathOperator{\Int}{Int}

\begin{document}
\title{Self-dual representations of $\Sp(4,F)$}
\author{Kumar Balasubramanian}
\thanks{Supported by DST-SERB Grant: YSS/2014/000806}

\maketitle
\begin{abstract}
Let $F$ be a non-Archimedean local field of characteristic $0$ and $G=\Sp(4,F)$. Let $(\pi,W)$ be an irreducible smooth self-dual representation $G$. The space $W$ of $\pi$ admits a non-degenerate $G$-invariant bilinear form $(\,,\,)$ which is unique up to scaling. The form $(\,,\,)$ is easily seen to be symmetric or skew-symmetric and we set $\varepsilon({\pi})=\pm 1$ accordingly. In this paper, we show that $\varepsilon{(\pi)}=1$ when $\pi$ is an Iwahori spherical representation of $G$. \\
\end{abstract}

\section{Introduction}

Let $G$ be a group and $(\pi,V)$ be an irreducible complex representation of $G$. Suppose that $\pi\simeq \pi^{\vee}$ ($\pi^{\vee}$ is the dual or contragredient representation). In the presence of Schur's lemma, it is easy to see that there exists a non-degenerate $G$-invariant bilinear form on $V$ which is unique up to scalars, and consequently is either symmetric or skew-symmetric. Accordingly, we set
\begin{equation*}
\varepsilon{(\pi)} =
\begin{cases}
\;\;\;1 & \text{if the form is symmetric},\\
-1 & \text{if the form is skew-symmetric},
\end{cases}
\end{equation*}
which we call the sign of $\pi$. In this paper, we study this sign for a certain class of representations of $\Sp(4,F)$. To be more precise, we show that $\varepsilon(\pi)=1$, when $\pi$ is an Iwahori-spherical representation of $G=\Sp(4,F)$. \\

The sign $\varepsilon{(\pi)}$ has been well studied for connected compact Lie groups and certain classes of finite groups of Lie type. If $G$ is a connected compact Lie group, it is known that the sign can be computed using the dominant weight attached to the representation $\pi$ (see \cite{BroTom} pg. 261-264). For finite groups of Lie type, computing the sign involves tedious conjugacy class computations. We refer to the following paper of Gow (\cite{Gow[2]}) where the sign is studied for such groups. In \cite{Pra[1]}, Prasad introduced an idea to compute the sign for a certain class of representations of finite groups of Lie type. He used this idea to determine the sign for many classical groups of Lie type, avoiding difficult computations. In recent times, there has been a lot of progress in studying these signs in the setting of reductive p-adic groups. In \cite{Pra[2]}, Prasad extended the results of \cite{Pra[1]} to the case of reductive p-adic groups and computed the sign of certain classical groups. The disadvantage of his method is that it works only for representations admitting a Whittaker model. In \cite{RocSpa}, Roche and Spallone discuss the relation between twisted sign (see section 1 in \cite{RocSpa}) and the ordinary sign and describe a way of studying the ordinary sign using the twisted sign. More recently in \cite{PraDin}, Prasad and Ramakrishnan have looked at signs of irreducible self-dual discrete series representations of $\GL_n(D)$, for $D$ a finite dimensional $p$-adic division algebra, and have proved a remarkable formula that relates the signs of these representations and the signs of their Langlands parameters. \\

In this paper, we compute $\varepsilon(\pi)$ for any irreducible smooth self-dual representation of $\Sp(4,F)$ with non-trivial vectors fixed under an Iwahori subgroup $I$. To be more precise, we prove the following

\begin{theorem}[Main Theorem] Let $G=\Sp(4,F)$ and $(\pi, W)$ be an irreducible smooth self-dual representation of $G$ with non-trivial vectors fixed under an Iwahori subgroup $I$. Then $\varepsilon(\pi)=1$.
\end{theorem}

We briefly explain the key ideas of the proof. We first consider the case when $\pi$ is a square integrable representation of $G$. It can be shown that such representations are generic. When $\pi$ is a generic representation of $G$, it is well known from the work of Prasad (proposition 2 in \cite{Pra[1]})) that the sign is determined by $\omega_{\pi}(-1)$, where $\omega_{\pi}$ is the central character of $\pi$. Using this along with the fact that $\pi$ is Iwahori spherical, it follows immediately that $\varepsilon(\pi)=1$. \\

In the case when $\pi$ is not generic, we consider the following two cases:\\
\begin{enumerate}
\item[a)] $\pi$ is not tempered: Here we use the results of Roche and Spallone (\cite{RocSpa}) and reduce the problem to computing the twisted sign (explained later) of a certain tempered representation of a Levi subgroup of $G$. Using the fact that the Levi subgroup of $G$ is of a certain type, allows us to apply some known results on Iwahori spherical representations to compute the sign. \\
\item[b)] $\pi$ is tempered: Here we compute the sign by reducing to the case of studying the twisted sign of a certain discrete series representation of a Levi subgroup of $G$. \\
\end{enumerate}
\vspace{0.3 cm}

The paper is organized as follows. In section~\ref{sign of a rep}, we introduce the notion of twisted and ordinary signs attached to the representation $\pi$. In section~\ref{main results}, we recall some results which we need. In the section \ref{main theorem} we prove the main theorem.

\section{Some preliminaries on signs}\label{sign of a rep}

In this section, we briefly discuss the notion of twisted and ordinary signs associated to representations.\\

Let $F$ be a non-Archimedean local field and $G$ be the group of $F$-points of a connected reductive algebraic group. Let $(\pi,W)$ be a smooth irreducible complex representation of $G$. We write $(\pi^{\vee}, W^{\vee})$ for the smooth dual or contragredient of $(\pi,W)$ and $\langle \, , \, \rangle$ for the canonical non-degenerate $G$-invariant pairing on $W \times W^{\vee}$ (given by evaluation). Let $\theta$ be a continuous automorphism of $G$ of order at most two. Let $(\pi^{\theta},W)$ be the $\theta$-twist of $\pi$ defined by \[\pi^{\theta}(g)w=\pi(\theta(g))w. \]
Suppose that $\pi^{\theta}\simeq \pi^{\vee}$. Let $s: (\pi^{\theta},W)\rightarrow (\pi^{\vee}, W^{\vee})$ be an isomorphism. The map $s$ can be used to define a bilinear form on $W$ as follows:
\begin{displaymath}
(w_{1}, w_{2})= \langle w_{1}, s(w_{2})\rangle, \quad \forall w_{1}, w_{2}\in W.
\end{displaymath}
It is easy to see that $(\,,\,)$ is a non-degenerate form on $W$ that satisfies the following invariance property
\begin{equation}\label{G-inv}
(\pi(g)w_{1}, \pi^{\theta}(g)w_{2})= (w_{1}, w_{2}), \quad \forall w_{1}, w_{2}\in W.\\
\end{equation}
Let $(\, , \,)_{\ast}$ be a new bilinear form on $W$ defined by
\begin{equation*}
(w_{1} , w_{2})_{\ast} = (w_{2}, w_{1})\\
\end{equation*}

Clearly, this form is again non-degenerate and $G$-invariant in the sense of \eqref{G-inv}. It follows from Schur's Lemma that
\begin{equation*}\label{dual}
(w_{1} , w_{2})_{\ast} = c(w_{1}, w_{2})
\end{equation*}
for some non-zero scalar $c$. A simple computation shows that $c\in \{\pm 1\}$. Indeed,
\begin{displaymath}
(w_{1},w_{2}) = (w_{2},w_{1})_{\ast} = c(w_{2},w_{1}) = c(w_{1},w_{2})_{\ast} = c^{2}(w_{1},w_{2}).
\end{displaymath}
We set $c=\varepsilon_{\theta}{(\pi)}$ and call it the twisted sign of $\pi$. It clearly depends only on the equivalence class of $\pi$. If $\theta$ is the trivial automorphism of $G$, we simply write $\varepsilon(\pi)$ instead of $\varepsilon_{1}(\pi)$ and call it the ordinary sign.  In sum, the form $(\,,\,)$ is symmetric or skew-symmetric and the sign $\varepsilon_{\theta}(\pi)$ determines its type.

\subsection{} Let $\theta$ be an automorphism of $G$ of order at most $2$ and suppose that $\pi^{\theta}\simeq \pi^{\vee}$. Consider the automorphism $\theta'$ of $G$ defined by
\[\theta'= \Int(h)\circ \theta\]
for $h\in G$, where $\Int(h)$ denotes the inner automorphism $g\rightarrow hgh^{-1}$ of $G$. In this situation, it is clear that $\pi^{\theta'}\simeq \pi^{\vee}$. Suppose $h\in G$ is chosen such that $\theta'$ has order at most two. It is easy to see that $h\theta(h)$ is a central element and $h\theta(h)=\theta(h)h$. A simple computation shows that
\begin{equation}\label{relation between signs}\varepsilon_{\theta'}(\pi)=\varepsilon_{\theta}(\pi)\omega_{\pi}(\theta(h)h)\end{equation} where $\omega_{\pi}$ is the central character. For specific details of this computation, we refer the reader to $1.2.1$ in \cite{RocSpa}.

\section{Some results we need}\label{main results}

In this section, we recall some results used in the proof of the main result of this paper.

\subsection{Prasad's Theorem}\label{Prasads theorem} Throughout this section we write $F_{q}$ for the finite field with $q$ elements. In \cite{Pra[1]}, Prasad shows that for an irreducible self-dual generic representation of $\Sp(n,F_{q})$, the sign $\varepsilon{(\pi)}$ is determined by the value of the central character $\omega_{\pi}$ of $\pi$. To be more precise, he proves the following

\begin{theorem}[Prasad] $\pi$ be an irreducible self-dual generic representation of $\Sp(n,F_{q})$. $\varepsilon{(\pi)}=1$ if and only if $\omega_{\pi}(-1)=1$.
\end{theorem}

Similar ideas along with the idea of compact approximation of whittaker models (\cite{Rod}, section III, pg. 155) can be used to prove the result when the finite field $F_{q}$ is replaced with a non-archimedean local field $F$. 

\subsection{Some results of Waldspurger}\label{Waldspurgers theorem}

In this section, we recall two results of Waldspurger which we need in the proof of the main theorem.

\subsubsection{} Throughout this section, we let $F$ be a non-Archimedean local field of characteristic $\neq 2$ and $W$ be a finite dimensional vector space over $F$. We let $\langle\, , \, \rangle$ to be a non-degenerate symmetric or skew-symmetric form on $W$. We take
\begin{displaymath}
G=\{g\in \GL(W) \mid \langle gw,gw' \rangle=\langle w, w' \rangle\}.
\end{displaymath}
For $x\in \GL(W)$ such that $xGx^{-1}=G$ and $(\pi,V)$ a representation of $G$, we let $\pi^{x}$ denote the representation of $G$ defined by conjugation (i.e., $\pi^{x}(g)=\pi(xgx^{-1})$.\\

We recall the statement of Waldspurger's theorem below and refer the reader to Chapter 4.II.1 in \cite{MoeWalVig} for a proof.

\begin{theorem}[Waldspurger]\label{waldspurger} Let $\pi$ be an irreducible admissible representation of $G$ and $\pi^{\vee}$ be the smooth-dual or contragredient of $\pi$. Let $x\in \GL(W)$ be such that $\displaystyle \langle xw,xw'\rangle = \langle w',w\rangle,\, \forall\, w,w'\in W$. Then $\pi^{x}\simeq \pi^{\vee}$.
\end{theorem}

\subsubsection{} Throughout this section, we take $G=G(F)$ and write $\mathcal{W}=W(G)$ for the Weyl group of $G$. For $H$ a subgroup of $G$, $s\in \mathcal{W}$, and $(\sigma,V)$ a representation of $H$, we let $^{s}\sigma(x)=\sigma(s^{-1}xs)$.

\begin{theorem}[Waldspurger]\label{waldspurger tempered} Let $\pi$ be an irreducible admissible tempered representation of $G$. Then:
\begin{enumerate}
\item[(i)] There exists a parabolic subgroup $P=MU$ and an irreducible admissible square integrable representation $\sigma$ of $M$ such that $\pi$ is a sub-quotient of $\ind_{P}^{G}\sigma$.
\item[(ii)] If $(P=MU, \sigma)$, $(P'=M'U', \sigma')$ are two representations of $G$ satisfying (i) above, then there exists $s\in \mathcal{W}$ such that $sMs^{-1}=M'$ and $^{s}\sigma=\sigma^{'}$.
\end{enumerate}
\end{theorem}

We refer the reader to III.4.1 in \cite{Wal[1]} for the proof of the above result.

%
%
%
%

\subsection{Reduction to Tempered case}\label{reduction} Throughout this section, we use the same notation and terminology as in \cite{RocSpa}. In \cite{RocSpa}, Roche and Spallone reduce the problem of computing the $\theta$-twisted sign to the case of tempered representations. We briefly recall their method below. For further details, we refer the reader to sections $\S 3, \S4$ of \cite{RocSpa}. \\

Let $\theta$ be an involutory automorphism of $G$ and suppose that $\pi^{\theta}\simeq \pi^{\vee}$. Let $(P, \tau, \nu)$ be the triple associated to $\pi$ via the Langlands' classification. Suppose that $P$ has Levi decomposition $P=MN$. Under certain assumptions on the involution $\theta$, they apply Casselman's pairing to show that $\varepsilon_{\theta}(\pi)=\varepsilon_{\theta}(\pi_{N})$, where $\pi_{N}$ is the Jacquet module of $\pi$. Using $\pi^{\theta}\simeq \pi^{\vee}$ and the fact that $\tau$ occurs with multiplicity one as a composition factor of $\pi_{N}$, they prove the following

\begin{theorem}[Roche-Spallone]\label{reduction to tempered case} Let $\pi$ be an irreducible smooth representation of $G$ such that $\pi^{\theta}\simeq \pi^{\vee}$. Suppose the Langlands' classification attaches the triple $(P, \tau, \nu)$ to $\pi$. Then $\tau^{\theta}\simeq \tau^{\vee}$ and $\varepsilon_{\theta}(\pi)=\varepsilon_{\theta}(\tau)$.
\end{theorem}

\begin{remark} In the above theorem, for $\tau^{\theta}$ to make sense, we need that $\theta$ preserves $M$, i.e., $\theta(M)=M$. This will follow from the assumptions on the involution $\theta$.
\end{remark}

\subsection{Restriction of representations}\label{restriction}
We recall some results about restricting an irreducible representation to a subgroup. These results hold when $G$ is a locally compact totally disconnected group and $H$ is an open normal subgroup of $G$ such that $G/H$ is (finite) abelian. For a more detailed account, we refer the reader to \cite{GelKna} (Lemmas 2.1 and 2.3).

\begin{theorem}[Gelbart-Knapp]\label{Gelbart-Knapp-1} Let $\pi$ be an irreducible admissible representation of $G$. Suppose that $G/H$ is (finite) abelian. Then
\begin{enumerate}
\item[(i)] $\pi|_{H}$ is a finite direct sum of irreducible admissible representations of $H$.
\item[(ii)] When the irreducible constituents of $\pi|_{H}$ are grouped according to their equivalence classes as
\end{enumerate}
\begin{equation*}
\pi|_{H}=\bigoplus_{i=1}^{M} m_{i}\pi_{i}
\end{equation*}
with the $\pi_{i}$ irreducible and inequivalent, the integers $m_{i}$ are equal.
\end{theorem}

\begin{theorem}[Gelbart-Knapp]\label{Gelbart-Knapp-2} Let $G$ be a locally compact totally disconnected group and $H$ be an open normal subgroup of $G$ such that $G/H$ is (finite) abelian, and let $\pi$ be an irreducible admissible representation of $H$. Then
\begin{enumerate}
\item[(i)] There exists an irreducible admissible representation $\tilde{\pi}$ of $G$ such that $\tilde{\pi}|_{H}$ contains $\pi$ as a constituent.
\item[(ii)] Suppose $\tilde{\pi}$ and $\tilde{\pi}'$ are irreducible admissible representations of $G$ whose restrictions to $H$ are multiplicity free and contain $\pi$. Then $\tilde{\pi}|_{H}$ and $\tilde{\pi}'|_{H}$ are equivalent and $\tilde{\pi}$ is equivalent with $\tilde{\pi}'\otimes \chi$ for some character $\chi$ of $G$ that is trivial on $H$.
\end{enumerate}
\end{theorem}

%
%
%
%
%
%
%
%
%
%
%
%
\section{Main Theorem}\label{main theorem}

Throughout this section, we set $G=\Sp(4,F)$ and write $I$ for an Iwahori subgroup in $G$. We let $\pi$ denote an irreducible smooth self-dual Iwahori spherical representation of $G$. We prove that $\varepsilon(\pi)=1$. \\

We consider the following cases to prove the main theorem:
\begin{enumerate}
\item[a)] $\pi$ is generic.
\item[b)] $\pi$ is square integrable.
\item[b)] $\pi$ is not generic.
\end{enumerate}

\subsection{$\pi$ is a generic representation} When $\pi$ is a generic representation of $G$, it is well known that $\varepsilon{(\pi)}=\omega_{\pi}(-1)$ (see \S 7, proposition 2 in \cite{Pra[1]}). Since $\pi$ is Iwahori spherical, it is easy to see that $\omega_{\pi}(-1)=1$ and the result follows. Indeed, for $0\neq v_{0}\in \pi^{I}$, we have \[v_{0}= \pi(-1)v_{0} = \omega_{\pi}(-1)v_{0}. \]

\subsection{$\pi$ is square integrable}

Throughout this section, we let $G=\Sp(4,F)$ and $\tilde{G}=\Gsp(4,F)$. We write $T$ (respectively $\tilde{T}$) for a maximal $F$-split torus in $G$ (respectively $\tilde{G}$). We write $I$ (respectively $\tilde{I}$) for an Iwahori subgroup in $G$ (respectively $\tilde{G}$). For $g\in \tilde{G}$, we let $\lambda_{g}\in F^{\times}$ denote the multiplier of $g$. It is easy to see that the map $g\rightarrow \lambda_{g}$ defines a homomorphism between $\tilde{G}$ and $F^{\times}$. We denote it by $\lambda$. \\

Let $(\pi,W)$ be an irreducible smooth self-dual square integrable representation of $G$ with non-trivial vectors fixed under an Iwahori subgroup $I$. We know that there exists an irreducible smooth square integrable representation $(\tilde{\pi},V)$ of $\tilde{G}$ such that the restriction of $\tilde{\pi}$ to $G$ contains the representation $\pi$ with multiplicity one. i.e., \[\tilde{\pi}|_{G}= \pi_{1}\oplus \cdots \oplus \pi_{k}\] with say $\pi_{1}\simeq \pi$. We can show that the representation $\tilde{\pi}$ can be chosen in such a way that $\tilde{\pi}^{\tilde{I}}\neq 0$. We do this by choosing a character $\tau_{1}$ of $\tilde{I}$ which is trivial on $I$ and extending it to a character $\tilde{\tau}$ of $\tilde{G}$ which is trivial on $G$ and considering the representation $\tilde{\pi}\tilde{\tau}^{-1}$ (see Lemma 4.11, Theorem 4.12 in \cite{Kum}). To simplify notation, we denote $\tilde{\pi}\tilde{\tau}^{-1}$ again as $\tilde{\pi}$. By construction, the character $\tilde{\tau}$ is a unitary character of $\tilde{G}$ and it follows that the modified representation $\tilde{\pi}$ is a square integrable representation which is also Iwahori spherical and contains the representation $\pi$ with multiplicity one on restriction to $G$. Since $\tilde{\pi}$ is an Iwahori spherical and a square integrable representation, it is generic (see \cite{SalTad}, \cite{Sch}). It follows that $\pi$ is also genric and by the previous case, we have $\varepsilon(\pi)=1$.

\subsection{$\pi$ is not a generic representation} We now consider the case when $\pi$ is not a generic representation of $G$. In this case, we prove the result by considering further classifications of the representation $\pi$. To be more precise, we consider the following cases separately\\

\begin{enumerate}
\item[a)] $\pi$ is not tempered.
\item[b)] $\pi$ is tempered (but not discrete series).\\
\end{enumerate}

Before we proceed further, we set up some notation and record a few observations. For vectors $v=(v_{1}, v_{2}, v_{3}, v_{4}), u=(u_{1}, u_{2}, u_{3}, u_{4}) \in F^{4}$, we let \[\langle v, u \rangle ^{'} = v_{1}u_{4} + v_{2}u_{3} - v_{3}u_{2} - v_{4}u_{1}\]
to be the skew-symmetric bilinear form used to define the symplectic group and let $w_{0}\in \GL(4,F)$ to be the matrix with anti-diagonal entries $1$. It is easy to see that $w_{0}^{2}=1$ and $\langle w_{0}v, w_{0}u \rangle ^{'} = \langle u , v \rangle ^{'}$. It is also clear that $w_{0}$ normalizes $G$. Indeed, for $g\in G$, we have
\begin{align*}
\langle w_{0}gw_{0}^{-1}v, w_{0}gw_{0}^{-1}u\rangle ^{'} &= \langle gw_{0}^{-1}u, gw_{0}^{-1}v \rangle ^{'}\\
&= \langle w_{0}^{-1}u, w_{0}^{-1}v \rangle ^{'}\\
&= \langle v, u \rangle  ^{'}.
\end{align*}

By Waldspurger's theorem (Chapter 4.II.1 in \cite{MoeWalVig}), it follows that $\pi^{\theta}\simeq \pi^{\vee}$. For $g\in \tilde{G}=\Gsp(4,F)$, we let $\iota(g)= \lambda_{g}^{-1}w_{0}gw_{0}^{-1}$, where $\lambda_{g}$ is the multiplier of $g$. It is easy to see that $\iota$ is a continuous automorphism of $\tilde{G}$ of order two and the restriction of $\iota$ to $G$ is $\theta$. Let $\tilde{\pi}$ be an irreducible smooth representation of $\tilde{G}$ such that the restriction of $\tilde{\pi}$ to $G$ contains the representation $\pi$ with multiplicity one, i.e., we have
\[\tilde{\pi}|_{G} = \pi_{1} \oplus \cdots \oplus \pi_{k},\]
with say $\pi_{1}\simeq \pi$. It can be shown that $\tilde{\pi}^{\iota}\simeq \tilde{\pi}^{\vee}$ (we refer the reader to Theorem B in \cite{RocVin} for a proof) and $\varepsilon_{\iota}(\tilde{\pi})$ makes sense. Since $\iota|_{G}=\theta$, and the restriction of the bilinear form on $\tilde{\pi}$ satisfying

\begin{equation}\label{form on pi tilde}
[\tilde{\pi}(g)(v), \tilde{\pi}^{\iota}(g)(w)] = [v,w]
\end{equation}

to $\pi$ is non degenerate, and we have
\begin{equation}\label{sign of iota andd theta}
\varepsilon_{\iota}(\tilde{\pi}) = \varepsilon_{\theta}(\pi).
\end{equation}

Consider the representation $\tilde{\pi}^{\vee}\simeq \tilde{\pi}\otimes \omega_{\tilde{\pi}}^{-1}$ (Here $\omega_{\tilde{\pi}}$ is the character obtained by composing the central character $\omega_{\tilde{\pi}}$ with the map $\lambda: \tilde{G}\rightarrow F^{\times}$). Since $\pi\simeq \pi^{\vee}$, it follows that the restriction of $\tilde{\pi}^{\vee}$ to $G$ also contains $\pi$ with multiplicity one. The bilinear form $(\,,\,)$ on $\tilde{\pi}$ satisfying
\[(\tilde{\pi}(g)v, \tilde{\pi}(g)w) = \omega_{\tilde{\pi}}(\lambda_{g})(v,w)\]
restricts to a non-degenerate $G$-invariant bilinear form on $\pi$ and it follows that
\begin{equation}\label{sign of pi tilde and pi}
\varepsilon(\tilde{\pi})=\varepsilon(\pi).
\end{equation}

We use the above results to establish a relation between the twisted sign $\varepsilon_{\theta}(\pi)$ and the ordinary sign $\varepsilon(\pi)$. We record the result in the following lemma.   \\

\begin{lemma} $\varepsilon_{\theta}(\pi) = \varepsilon(\pi)\omega_{\pi}(-1)$.
\end{lemma}
\begin{proof} Consider the non-degenerate bilinear form on $\tilde{\pi}$ defined by \[[v_{1}, v_{2}]'=(v_{1}, \tilde{\pi}(w_{0}^{-1})v_{2}).\] It is clearly $\tilde{G}$ invariant in the sense of~\eqref{form on pi tilde}. Indeed,
\begin{align*}
[\tilde{\pi}(g)v_{1}, \tilde{\pi}^{\iota}(g)v_{2}]' &= (\tilde{\pi}(g)v_{1}, \tilde{\pi}(w_{0}^{-1})\tilde{\pi}^{\iota}(g)v_{2})\\
&= (\tilde{\pi}(g)v_{1}, \tilde{\pi}(w_{0}^{-1})\tilde{\pi}(\lambda_{g}^{-1}w_{0}gw_{0}^{-1})v_{2})\\
&= (\tilde{\pi}(g)v_{1}, \tilde{\pi}(w_{0}^{-1}\lambda_{g}^{-1}w_{0}gw_{0}^{-1})v_{2})\\
&= (\tilde{\pi}(g)v_{1}, \tilde{\pi}(\lambda_{g}^{-1})\tilde{\pi}(gw_{0}^{-1})v_{2})\\
&= \omega_{\tilde{\pi}}(\lambda_{g}^{-1})(\tilde{\pi}(g)v_{1}, \tilde{\pi}(gw_{0}^{-1})v_{2})\\
&= \omega_{\tilde{\pi}}(\lambda_{g}^{-1})\omega_{\tilde{\pi}}(\lambda_{g})(v_{1}, \tilde{\pi}(w_{0}^{-1})v_{2})\\
&= [v_{1}, v_{2}]'.
\end{align*}

A simple computation shows that \[[v_{1},v_{2}]' = \varepsilon_{\iota}(\tilde{\pi})[v_{2},v_{1}]' = \varepsilon(\tilde{\pi})\omega_{\tilde{\pi}}(-1).\]
Indeed,
\begin{align*}
[v_{1},v_{2}]' &= \varepsilon_{\iota}(\tilde{\pi})[v_{2},v_{1}]'\\
&= (v_{1}, \tilde{\pi}(w_{0}^{-1})v_{2})\\
&= \varepsilon(\tilde{\pi})(\tilde{\pi}(w_{0}^{-1})v_{2}, v_{1})\\
&= \varepsilon(\tilde{\pi})\omega_{\tilde{\pi}}(\lambda_{w_{0}}^{-1})(v_{2}, \tilde{\pi}(w_{0})v_{1})\\
&= \varepsilon(\tilde{\pi})\omega_{\tilde{\pi}}(-1)(v_{2}, \tilde{\pi}(w_{0}^{-1})v_{1})\\
&= \varepsilon(\tilde{\pi})\omega_{\tilde{\pi}}(-1)[v_{2},v_{1}]'.
\end{align*}
The result now follows from ~\eqref{sign of iota andd theta} and ~\eqref{sign of pi tilde and pi}.

\end{proof}

Since $\pi$ has non trivial Iwahori fixed vectors, it follows that $\omega_{\pi}(-1)$ is trivial and hence \[\varepsilon_{\theta}(\pi)=\varepsilon(\pi).\]

We now focus on computing the twisted sign attached to $\pi$.

\subsubsection{$\pi$ is not tempered} Let $(P, \tau, \nu)$ be the triple associated to $\pi$ via the Langlands' classification. We let $M$ and $N$ denote the Levi component and the unipotent radical of the parabolic subgroup $P$. Before we proceed further, we observe that the automorphism $\theta$ satisfies the hypotheses needed in order to apply \textit{Casselman's pairing} as stated in \S 3 of \cite{RocSpa}. To be more precise, we have

\begin{lemma}\label{properties of the involution theta'} The involution $\theta$ satisfies the following conditions
\begin{enumerate}
\item[(i)] $\theta$ is an automorphism of $G$ as an algebraic group.
\item[(ii)] $\theta$ preserves $T$ so that $\theta|_{T}$ is an involutory automorphism of the $F$-split torus $T$ and
\item[(iii)] $\theta$ maps $N$ to the $M$-opposite $\bar{N}$.
\end{enumerate}
\end{lemma}

\begin{proof} (i) and (ii) are clearly satisfied. For (iii), since $\pi\hookrightarrow\ind_{P}^{G}(\tau\nu)$, it follows (by taking duals) that $\pi^{\vee}$ is a quotient of $\ind_{P}^{G}(\tau^{\vee}\nu^{-1})$. In other words, we have \begin{equation}\label{langlands triple of pi eq 1}\pi^{\vee}\hookrightarrow \ind_{\bar{P}}^{G}(\tau^{\vee}\nu^{-1}). \end{equation}
Since $\pi\simeq \pi^{\theta}\simeq \pi^{\vee}$, it follows that \begin{equation}\label{langlands triple of pi eq 2}\pi^{\vee}\hookrightarrow \ind_{\theta(P)}^{G}(\tau^{\theta}\nu^{\theta}).\end{equation}
From \eqref{langlands triple of pi eq 1} and \eqref{langlands triple of pi eq 2}, and the uniqueness of the Langlands' classification, it follows that $\bar{P}=\theta(P)$ and $\tau^{\theta}\simeq \tau^{\vee}$. In particular, we have $\theta(N)=\bar{N}$.
\end{proof}

From Lemma \ref{properties of the involution theta'} and Theorem \ref{reduction to tempered case}, it follows that
\begin{equation}\label{twisted sign of tau}
\varepsilon_{\theta}(\pi)=\varepsilon_{\theta}(\tau).
\end{equation}

Throughout we write $I_{M}=I\cap M$. Before we continue, we observe that $\tau$ has nontrivial $I_{M}$ fixed vectors. We record the result in the following

\begin{lemma} The representation $\tau$ has non-trivial $I_{M}$ fixed vectors.
\end{lemma}
\begin{proof} Since $\pi\hookrightarrow \ind_{P}^{G}(\tau\nu)$ and $\pi^{I}\neq 0$, it follows that $(\pi_{N})^{I_{M}}\neq 0$. Since $\tau\nu$ occurs as a composition factor of $\pi_{N}$, it follows that $(\tau\nu)^{I_{M}}\neq 0$ (refer Lemma 4.7 and Lemma 4.8 in \cite{Bor}). Now using the fact that $I_{M}$ is compact and $\nu$ is a (continuous) character of $M$ taking positive real values, it is clear that $\nu|_{I_{M}}=1$ and $\tau^{I_{M}}\neq 0$.
\end{proof}

From \eqref{twisted sign of tau}, it is enough to compute the sign $\varepsilon_{\theta}(\tau)$ where $\tau$ is an irreducible tempered representation of $M$ with $\tau^{I_{M}}\neq 0$. Since $M$ is a Levi subgroup of $G$, we have the following possibilities for $M$: \\
\begin{enumerate}
\item[(a)] $M\simeq \GL(1,F) \times \GL(1,F)$\\
\item[(b)] $M\simeq \GL(2,F)$\\
\item[(c)] $M\simeq \SL(2,F) \times \GL(1,F)$
\end{enumerate}


\vspace{0.2 cm}

In all these cases, we show that $\varepsilon_{\theta}(\tau)=1$. \\


 For cases (a) and (b), considering $\theta$ as an automorphism of $\GL(1,F)\times \GL(1,F)$ and $\GL(2,F)$, it is easy to see that $\theta(g)=^{\top}g^{-1}$ (Here $\top$ denotes the transpose). Now using $\tau^{\theta}\simeq \tau^{\vee}$ (as a representation of either $\GL(1,F)\times \GL(1,F)$ or $\GL(2,F)$), and $\tau$ is tempered (hence generic) (see \cite{Jac}, \cite{JacPiaSha}), it follows that $\varepsilon_{\theta}(\tau)=1$ (\S 5, Theorem 1 in \cite{RocSpa}).\\

We now consider the case when $M\simeq \SL(2,F)\times \GL(1,F)$. Since $\tau$ is an irreducible tempered representation of $M$, $\tau\simeq \tau_{1}\otimes \tau_{2}$ where $\tau_{1}$, $\tau_{2}$ are irreducible tempered representations of $\SL(2,F)$ and $\GL(1,F)$ respectively. Since $\tau^{\theta}\simeq \tau^{\vee}$, we have
    \begin{displaymath}
    (\tau_{1}\otimes \tau_{2})^{\theta} \simeq \tau_{1}^{\theta} \otimes \tau_{2}^{\theta} \simeq (\tau_{1}\otimes \tau_{2})^{\vee} \simeq \tau_{1}^{\vee} \otimes \tau_{2}^{\vee}.
    \end{displaymath}
\vspace{0.2 cm}

It is clear that $\tau_{1}^{\theta}\simeq \tau_{1}^{\vee}$ (here $\theta$ acts like the involution $\theta(g) = xgx^{-1}$, where $x$ is the matrix with anti diagonal entries $1$) and $\tau_{2}^{\theta}\simeq \tau_{2}^{\vee}$ (here $\theta$ acts like the involution $\theta(g)=^{\top}g^{-1}$). It is also easy to see that $I_{M}=I\cap M \simeq I_{\SL(2,F)} \times I_{\GL(1,F)}$. Since $\tau^{I_{M}}\neq 0$, it follows that $\tau_{1}^{I_{\SL(2,F)}}\neq 0$ and $\tau_{2}^{I_{\GL(1,F)}}\neq 0$.\\ 

Clearly $\varepsilon_{\theta}(\tau_{2})=1$ (\S 5, Theorem 1 in \cite{RocSpa}). Now consider the representation $\tau_{1}$. Since $\tau_{1}$ is a tempered representation of $\SL(2,F)$, it is generic (follows from Lemma 4.3 in \cite{Kum[2]}). We know that there exists an irreducible smooth representation $\tilde{\tau}_{1}$ of $\GL(2,F)$ such that $\tilde{\tau}_{1}|_{\SL(2,F)}\supset \tau_{1}$ with multiplicity one. We can in fact, choose $\tilde{\tau}_{1}$ such that $\tilde{\tau}_{1}^{I_{\GL(2,F)}}\neq 0$ (see Lemma 4.11 and Theorem 4.12 in \cite{Kum}). It is clear that $(\tilde{\tau}_{1}^{\theta})^{\vee}$ also contains the representation $\tau_{1}$ with multiplicity one. It follows that \begin{equation}\label{representation tau tide}(\tilde{\tau}_{1}^{\theta})^{\vee}\simeq \tilde{\tau}_{1} \otimes \chi,\end{equation} where $\chi$ is a character of $\GL(2,F)$ such that $\chi|_{\SL(2,F)}=1$ (see Lemmas 2.1 and 2.3 in \cite{GelKna}). Using ~\eqref{representation tau tide}, it is easy to see that the space of $\tilde{\tau}_{1}$ admits a non-degenerate bilinear form $[\, , \,]$ satisfying
\begin{equation}\label{sign of tau tilde}
[\tilde{\tau}_{1}(g)v_{1}, (\tilde{\tau}_{1}^{\theta})(g)v_{2}]= \chi^{-1}(\theta(g))[v_{1}, v_{2}]
\end{equation}
It is easy to see that the above form $[\, , \,]$ is unique up to scalars (hence symmetric or skew-symmetric) and we can attach a sign $\varepsilon_{\theta}^{\chi}(\tilde{\tau}_{1}) \in \{\pm 1\}$ capturing the nature of the form. To simplify notation, we write $\varepsilon_{\theta}^{\chi}(\tilde{\tau}_{1})= \varepsilon_{\theta}(\tilde{\tau}_{1})$. We can show that the form $[\, , \,]|_{\tau_{1}\times \tau_{1}}$ is non-degenerate and satisfies
\begin{equation}\label{sign of tau}
[\tau_{1}(g)w_{1}, \tau_{1}^{\theta}(g)w_{2}] = [w_{1}, w_{2}].
\end{equation}
We refer the reader to Lemma 4.14 and Lemma 4.15 in \cite{Kum} for details. It follows from ~\eqref{sign of tau tilde} and ~\eqref{sign of tau} that
\begin{equation}\label{sign of tau tilde is sign of tau}\varepsilon_{\theta}(\tilde{\tau}_{1})=\varepsilon_{\theta}(\tau_{1}).\end{equation}
For $g\in \GL(2,F)$, we let $\alpha(g)= (\Int(x)\circ\theta)(g) =  x\theta(g)x^{-1}$. It is clear that $\alpha(g)=1$ and  \[(\tilde{\tau}_{1}^{\theta})^{\vee}=(\tilde{\tau}_{1}^{\alpha})^{\vee}\simeq \tilde{\tau}_{1}\otimes \chi.\]
A simple computation shows that
\begin{equation}\label{twisted sign of alpha and theta}\varepsilon_{\alpha}(\tilde{\tau}_{1})= \varepsilon_{\theta}(\tilde{\tau}_{1})\omega_{\tilde{\tau}_{1}}(\theta(x)x). \end{equation}
We refer the reader to \S 1.2 in \cite{RocSpa} for the details. Since $\alpha(g)=1$ and $\theta(x)x=1$, it follows that \[ \varepsilon_{\alpha}(\tilde{\tau}_{1})= \varepsilon_{\theta}(\tilde{\tau}_{1})\omega_{\tilde{\tau}_{1}}(\theta(x)x) = \varepsilon_{\theta}(\tilde{\tau}_{1}).\]
\vspace{0.2 cm}
It follows from earlier work that $\varepsilon_{\alpha}(\tilde{\tau_{1}})= 1$ (see \S 4.2.5 in \cite{Kum}) and the result follows.

\subsubsection{$\pi$ is tempered} Throughout this section, we let $W(G)$ denote the Weyl group of $G$. We have \[W(G)=\{1, s_{\alpha}, s_{\beta}, s_{\alpha}s_{\beta}, s_{\beta}s_{\alpha}, s_{\alpha}s_{\beta}s_{\alpha}, s_{\beta}s_{\alpha}s_{\beta}, (s_{\beta}s_{\alpha})^{2}\},\] where
\[s_{\alpha}=\begin{bmatrix*} 0 & 1 & 0 & 0 \\ 1 & 0 & 0 & 0 \\ 0 & 0 & 0 & 1\\ 0 & 0 & 1 & 0\end{bmatrix*}; \quad s_{\beta}= \begin{bmatrix*} 1 & 0 & 0 & 0 \\ 0 & 0 & 1 & 0 \\ 0 & -1 & 0 & 0\\ 0 & 0 & 0 & 1\end{bmatrix*}.\]
\vspace{0.3 cm}


Since $\pi$ is tempered, it follows from part (i) of Waldspurger's theorem (see Theorem~\ref{waldspurger tempered} above), that there exists a parabolic subgroup $P=MN$, and an irreducible smooth discrete series representation $\sigma$ of $M$ such that $\pi\hookrightarrow \ind_{P}^{G}\sigma$. We first show that $^{x}\sigma\simeq (\sigma^{\theta})^{\vee}$ for some $x\in W(G)$. We record the result in the following lemma.

\begin{lemma} There exists $x\in W(G)$ such that $^{x}\sigma\simeq (\sigma^{\theta})^{\vee}$.
\end{lemma}

\begin{proof} Since $\pi\hookrightarrow \ind_{P}^{G}\sigma$, we have $\pi^{\theta}\hookrightarrow \ind_{\theta(P)}^{G}(\sigma^{\theta})$. Since $\theta(P)=\bar{P}$, it follows that $\pi^{\theta}\hookrightarrow \ind_{\bar{P}}^{G}(\sigma^{\theta})$. Now taking dual, replacing $\bar{P}$ by $P$, and using $(\pi^{\theta})^{\vee}\simeq \pi$ we get $\pi\hookrightarrow \ind_{P}^{G}((\sigma^{\theta})^{\vee})$). Applying part (ii) of Waldspurger's theorem (see Theorem~\ref{waldspurger tempered} above), the result follows.
\end{proof}

Since $\sigma$ is a discrete series representation, it follows that $\pi$ occurs with multiplicity one in $\ind_{P}^{G}\sigma$ (see \cite{Gol}, \cite{Her}). As observed before, $\theta$ satisfies all the hypotheses needed in order to apply Casselman's pairing (see \S 3 of \cite{RocSpa}) and it follows that \begin{equation}\label{sign of pi and jacqeut module of pi}\varepsilon_{\theta}(\pi)=\varepsilon_{\theta}(\pi_{N}),\end{equation} where $\pi_{N}$ is the Jacquet module of $\pi$ with respect to $P$. \\

Let $\theta'(m)=\theta(x^{-1}mx)$. Since $^{x}\sigma\simeq (\sigma^{\theta})^{\vee}$ and $w_{0}^{2}=1$, it is clear that $\sigma^{\theta'}\simeq \sigma^{\vee}$. We first conclude that $x\not \in \{s_{\alpha}s_{\beta}, s_{\beta}s_{\alpha}\}$. Indeed, if $M\simeq \GL(1,F)\times \GL(1,F)$, and $x\in \{s_{\alpha}s_{\beta}, s_{\beta}s_{\alpha}\}$, then $\sigma$ is the trivial character and hence not a discrete series representation. In the case when $M\simeq \GL(2,F)$ or $M\simeq \SL(2,F)\times GL(1,F)$, it is easy to see that these elements don't normalize $M$. A simple computation shows that if $x\not \in \{s_{\alpha}s_{\beta}, s_{\beta}s_{\alpha}\}$ then $(xw_{0})^{2}=1$ and $\theta'= (int(x)\circ \theta)$. Indeed,
\begin{align*}
\theta'(m) &= \theta(x^{-1}mx)\\
&= w_{0}x^{-1}mxw_{0}^{-1}\\
&= xw_{0}mw_{0}^{-1}x^{-1}\\
&= x\theta(m)x^{-1}\\
&= (int(x)\circ \theta)(m).
\end{align*}

Since $\theta'= (int(x)\circ \theta)$ and $\theta$ is an involution, it is clear that $\theta'$ is an involution and $\sigma^{\theta}\simeq \sigma^{\theta'}$. It is easy to see that $\theta(x)x=1$ and it follows from 1.2.1 in \cite{RocSpa} that
\begin{equation}\label{sign of theta is sign if theta prime}
\varepsilon_{\theta'}(\sigma) = \varepsilon_{\theta}(\sigma)
\end{equation}

Since $\sigma^{\theta}\simeq \sigma^{\theta'}\simeq \sigma^{\vee}$ and $\sigma^{\theta}$ appears with multiplicity one in $\pi_{N}$, corollary in \S 3.5 in \cite{RocSpa}, applies and we have \begin{equation}\label{sign of sigma is sign of the jacquet module}\varepsilon_{\theta}(\pi_{N})=\varepsilon_{\theta}(\sigma).\end{equation}

In the case when $M\simeq \GL(2,F)$ or $M\simeq \GL(1,F)\times \GL(1,F)$, we know that $\theta$ acts as the involution $\theta(g)=(^{\top}g)^{-1}$ and it is clear that $\varepsilon_{\theta}(\sigma)=1$ (see \S 6 Theorem 1 in \cite{RocSpa}) and hence $\varepsilon(\pi)=1$.\\

If $M\simeq \SL(2,F)\times \GL(1,F)$, we have $\sigma\simeq \sigma_{1}\otimes \sigma_{2}$, where $\sigma_{1}$ and $\sigma_{2}$ are irreducible, Iwahori spherical, discrete series representations of $\SL(2,F)$ and $\GL(1,F)$ respectively. Since these representations are generic, it follows from the tempered case (considered earlier) that $\varepsilon_{\theta}(\sigma)=1$ and hence $\varepsilon(\pi)=1$.\\

\section*{Acknowledgements}
I would like to thank professors Dipendra Prasad and Alan Roche for their many valuable suggestions, help and encouragement throughout this project.
%
%

\bibliographystyle{amsplain}
\bibliography{ref}
\vspace{1 cm}
Department of Mathematics, \\
Indian Institute of Science Education and Research Bhopal, \\
Bhopal 462066.\\
\textit{E-mail address}: bkumar@iiserb.ac.in
\end{document}